\documentclass[a4paper,12pt]{article}

\usepackage[utf8]{inputenc}
\usepackage[english]{babel}
\usepackage{enumitem}
\usepackage{float}
\usepackage{amsmath}
\usepackage{amssymb}
\usepackage{amsthm}
\usepackage{amsfonts}
\usepackage{booktabs}
\usepackage{epsfig}
\usepackage{epstopdf}
\usepackage{cite}
\usepackage{multirow} 
\usepackage{algorithm,algpseudocode}
\usepackage{subcaption}
\usepackage{graphicx}
\usepackage{subcaption}
\usepackage{xcolor}
\usepackage[linkcolor=blue,colorlinks=true]{hyperref}
\let\originaleqref\eqref
\renewcommand{\eqref}{~\originaleqref}
\newtheorem{theorem}{Theorem}[section]
\newtheorem{lemma}[theorem]{Lemma}

\newtheorem{definition}[theorem]{Definition}

\newtheorem{remark}[theorem]{Remark}
\newtheorem{assumption}{Assumption}[section]
\usepackage{caption}
\renewcommand{\xi}{y}
\renewcommand{\psi}{y}
\renewcommand{\chi}{x}
\renewcommand{\gamma}{x}

\allowdisplaybreaks

\title{Fractional-Order Nesterov Dynamics for Convex Optimization}

\date{\today}

\author{
    Tumelo Ranoto \\
    \texttt{r.ranoto.tt@gmail.com}
}
\begin{document}

\maketitle
\begin{abstract}
\noindent We propose and analyze a class of second-order dynamical systems for continuous-time optimization that incorporate fractional-order gradient terms. The system is given by
\begin{equation}
\ddot{x}(t) + \frac{\alpha}{t}\dot{x}(t)  + \nabla^{\theta} f(x(t)) = 0,
\end{equation}
where $\theta \in (1,2)$, and the fractional operators are interpreted in the sense of Caputo, Riemann--Liouville, and Grünwald--Letnikov derivatives. This formulation interpolates between memory effects of fractional dynamics and higher-order damping mechanisms, thereby extending the classical Nesterov accelerated flow into the fractional domain. A particular focus of our analysis is the regime $\alpha \leq 3$, and especially the critical case $\alpha = 3$, where the ordinary Nesterov flow fails to guarantee convergence. We show that in the fractional setting, convergence can still be established, with fractional gradient terms providing a stabilizing effect that compensates for the borderline damping. This highlights the ability of fractional dynamics to overcome fundamental limitations of classical second-order flows. We develop a convergence analysis framework for such systems by introducing fractional Opial-type lemmas and Lyapunov memory functionals. In the convex case, we establish weak convergence of trajectories toward the minimizer, as well as asymptotic decay of functional values. For strongly convex functions, we obtain explicit convergence rates that improve upon those of standard second-order flows. \\


\end{abstract}

\section{Introduction}\label{Sec:Intro}
\noindent
Optimization plays a fundamental role in various fields, including machine learning, data science, control systems, and signal processing. A general optimization problem is typically formulated as:
\begin{equation}
\min_{x \in  H} f(x) \quad \text{s.t.} \quad x \in C,
\end{equation}
where $x$ is the decision variable, $f: H \to C$ represents the objective function, and $C \subseteq H$ is the feasible set defined by constraints. The solution to this problem involves finding the point within $C$ that minimizes $f(x)$. In recent years, continuous-time dynamical systems have emerged as powerful tools for studying and designing optimization algorithms. These systems, especially first and second-order differential equations, offer a rich framework for analyzing convergence properties and algorithmic behavior~\cite{attouch2013convergence, su2016differential, wibisono2016variational}. With the growing prevalence of complex applications in machine learning and image processing, the use of nonconvex and nonsmooth objective functions has become increasingly common. While nonconvex formulations offer more flexibility in modeling real-world problems, they introduce substantial challenges due to their NP-hard nature~\cite{Machine1}. This shift in problem structure demands new theoretical tools and algorithmic techniques that go beyond the traditional assumptions of smoothness and convexity. \\

\noindent Classical continuous-time approaches, such as gradient flow systems  
\begin{equation}
\dot{x}(t) = -\nabla f(x(t)),
\end{equation}  
and second-order inertial systems  
\begin{equation}
\ddot{x}(t) + \frac{\alpha}{t} \dot{x}(t) + \nabla f(x(t)) = 0,
\end{equation}  
have played a central role in both theoretical analysis and the design of efficient numerical algorithms. These ordinary differential equation (ODE)-based models offer valuable insights into convergence rates, stability properties, and asymptotic behaviors of optimization algorithms. Other classical works, such as Polyak’s heavy ball method~\cite{Polyak1987}, have established deep connections between optimization and second-order dynamical systems, leading to sharp convergence guarantees. More recent developments introduced Hessian-driven damping~\cite{Attouch2016}, implicit regularization~\cite{Azouzi2021}, and higher-order systems~\cite{AttouchChbaniRiahi2022}, significantly advancing our understanding of stability and convergence in convex and nonconvex optimization. Despite these advances, existing approaches often rely on fixed or asymptotically vanishing damping mechanisms, which may fail to adapt effectively in high-dimensional or non-stationary landscapes.\\

\noindent 
In recent years, fractional calculus, the study of derivatives and integrals of non-integer order, has emerged as a powerful tool for modeling dynamical processes with memory and hereditary effects \cite{Aggarwal2023_ConvFracGD,VieiraRodriguesFerreira2023_PsiHilfer}. In contrast to classical derivatives, fractional operators such as those of Riemann--Liouville, Caputo, and Riesz naturally encode long-range temporal dependencies, allowing the system's present state to depend not only on the instantaneous gradient but also on its entire history of evolution. This inherent nonlocality has motivated the development of fractional-order optimization methods, including Caputo fractional gradient descent \cite{ShinDarbonKarniadakis2021_CaputoOptimization}, $\psi$-Hilfer based gradient schemes \cite{VieiraRodriguesFerreira2023_PsiHilfer}, and fractional variants of Adam and accelerated flows \cite{ShinDarbonKarniadakis2023_FracGradientsAdam,C-FOG2024_SelfOrganizing}. Such approaches have been shown to mitigate oscillations near saddle points, escape narrow valleys more effectively, and accelerate convergence compared to classical integer-order methods.\\

\noindent 
The idea of combining fractional differential equations (FDEs) with continuous-time optimization arises from the observation that many real-world systems, ranging from anomalous diffusion in physics to viscoelastic deformation in materials science, cannot be accurately modeled using only integer-order dynamics \cite{MetzlerKlafter2000_AnomalousDiffusion,Mainardi2010_Viscoelasticity,Aggarwal2023_ConvFracGD,VieiraRodriguesFerreira2023_PsiHilfer}. Extending gradient-based flows to their fractional counterparts allows for richer search dynamics, improved ability to escape local minima, and tunable convergence behaviors \cite{ShinDarbonKarniadakis2021_CaputoOptimization,ShinDarbonKarniadakis2023_FracGradientsAdam}. A prototypical example is the fractional gradient flow
\begin{equation}
D_t^\alpha x(t) = -\nabla f(x(t)), \quad 0<\alpha \leq 1,
\end{equation}  
where $D_t^\alpha$ denotes a fractional derivative in time. This model generalizes the classical steepest descent method, recovering it as the special case $\alpha=1$, while for $\alpha<1$ it incorporates long-memory effects that can stabilize trajectories and accelerate convergence in challenging landscapes \cite{C-FOG2024_SelfOrganizing}.\\

\noindent
Several recent studies have demonstrated that fractional‐order dynamics can improve robustness against noisy gradients, enhance global exploration in nonconvex problems, and offer superior convergence properties under suitable conditions \cite{Aggarwal2023_ConvFracGD,ShinDarbonKarniadakis2023_FracGradientsAdam,RobustInitFgd2022,IFOGD2024_HiddenLayers}. Moreover, the continuous‐time analysis of such systems provides a rigorous bridge between fractional‐order differential equations and their discrete algorithmic implementations, leading to novel fractional gradient descent methods with provable convergence guarantees \cite{Aggarwal2023_ConvFracGD,ShinDarbonKarniadakis2023_FracGradientsAdam}.\\

\noindent  In this work, we develop and analyze fractional continuous dynamical systems for solving optimization problems, unifying the modeling flexibility of fractional calculus with the theoretical framework of optimization. Our study investigates both convex and nonconvex settings, focusing on the interplay between memory effects, convergence rates, and stability, and aims to provide guidelines for selecting fractional orders that balance fast convergence with robustness. In particular, we introduce a rigorous framework for fractional Nesterov flows: by employing fractional Opial-type lemmas and Lyapunov--memory functionals, we establish weak convergence of trajectories for convex objectives and convergence to critical points for nonconvex problems. Furthermore, in the strongly convex case, we derive explicit convergence rates and show that fractional acceleration can surpass classical second-order methods in both speed and stability. The proposed framework provides a solid continuous-time foundation for designing new classes of accelerated optimization algorithms that leverage fractional dynamics, with potential applications in machine learning, signal processing, and other areas requiring efficient minimization of complex objectives.\\

\noindent
\textbf{Organization}.
\noindent We organize the rest of our paper as follows:  Section \ref{Sec:Prelims} contains basic definitions and results needed in subsequent sections. In Section \ref{chapter3}, we present and discuss our proposed method with the existence and uniqueness of global convergence. In Section~\ref{chapter4}, we carry out the convergence analysis for the case where the objective function is convex and obtain weak convergence. Furthermore, in Section~\ref{chapter4}, we establish the Mittag-Leffler rate in the case where the objective function is strongly convex and Lipschitz continuous. We present concluding remarks in Section \eqref{conclusion}.

\section{Preliminaries}\label{Sec:Prelims}

This section reviews essential concepts from fractional calculus and continuous optimization, and introduces the novel analytical tools required for our analysis.

\subsection{Fractional Calculus Definitions}

We begin by defining the fractional operators central to our work. Let $x : [0, T] \to H$ be an absolutely continuous function, and let $\theta > 0$ be a real number with $n-1 < \theta < n$ for some $n \in \mathbb{N}$. Let $\Gamma(\cdot)$ denote the Gamma function.

\begin{definition}[Riemann--Liouville Fractional Integral]
The left-sided Riemann--Liouville integral of order $\theta$ is defined as:
\begin{equation}
\label{eq:RL-integral}
I^{\theta} x(t) = \frac{1}{\Gamma(\theta)} \int_0^t (t-s)^{\theta-1} x(s)  ds.
\end{equation}
\end{definition}

\begin{definition}[Riemann--Liouville Fractional Derivative]
The left-sided Riemann--Liouville derivative of order $\theta$ is defined as:
\begin{equation}
\label{eq:RL-derivative}
_{RL}D^{\theta}_t x(t) = \frac{d^n}{dt^n} \left[ I^{n-\theta} x(t) \right] = \frac{1}{\Gamma(n-\theta)} \frac{d^n}{dt^n} \int_0^t (t-s)^{n-\theta-1} x(s)  ds.
\end{equation}
\end{definition}

\begin{definition}[Caputo Fractional Derivative]
The left-sided Caputo derivative of order $\theta$ is defined as:
\begin{equation}
\label{eq:Caputo-derivative}
_{C}D^{\theta}_t x(t) = I^{n-\theta} \left[ x^{(n)}(t) \right] = \frac{1}{\Gamma(n-\theta)} \int_0^t (t-s)^{n-\theta-1} x^{(n)}(s)  ds.
\end{equation}
The Caputo derivative is often preferred in modeling physical systems as it allows for standard initial conditions of the form $x(0)=x_0, \dot{x}(0)=v_0, \dots$.
\end{definition}

\begin{definition}[Grünwald--Letnikov Fractional Derivative]
The left-sided Grünwald--Letnikov derivative is defined as the limit of a fractional difference quotient:
\begin{equation}
\label{eq:GL-derivative}
_{GL}D^{\theta}_t x(t) = \lim_{h \to 0^+} h^{-\theta} \sum_{j=0}^{\lfloor t/h \rfloor} (-1)^j \binom{\theta}{j} x(t - jh).
\end{equation}
This definition highlights the discrete, memory-heavy nature of fractional derivatives and provides a direct link to numerical implementation.
\end{definition}

\noindent For sufficiently smooth functions $x(t)$ (e.g., $x \in C^n[0, T]$), the Riemann-Liouville and Caputo derivatives are related by:
\begin{equation}
\label{eq:RL-Caputo-relation}
_{C}D^{\theta}_t x(t) = _{RL}D^{\theta}_t x(t) - \sum_{k=0}^{n-1} \frac{t^{k-\theta}}{\Gamma(k-\theta+1)} x^{(k)}(0).
\end{equation}
Under the same smoothness conditions, the Grünwald--Letnikov derivative coincides with the Riemann--Liouville derivative. The fundamental property of all these definitions is \textit{non-locality} (or \textit{memory}): the derivative at time $t$ depends on the entire history of the function from $0$ to $t$.

\subsection{Function Classes and Assumptions}

We consider the problem of minimizing a function $f : H \to C$, and make the following standard assumptions throughout this work.

\begin{assumption}[Basic Regularity]
The function $f$ is continuously differentiable ($f \in C^1(\mathbb{R}^n)$), bounded below, and its gradient $\nabla f$ is $L$-Lipschitz continuous:
\begin{equation}
\| \nabla f(x) - \nabla f(y) \| \leq L \| x - y \| \quad \forall x, y \in H.
\end{equation}
We denote the set of minimizers by $X^* = \{ x^* \in \mathbb{R}^n : f(x^*) = \inf f \}$ and assume it is non-empty.
\end{assumption}

\noindent Our convergence results will be established under the following additional convexity properties.

\begin{assumption}[Convexity]
The function $f$ is convex, i.e.,
\begin{equation}
f(y) \geq f(x) + \langle \nabla f(x), y - x \rangle \quad \forall x, y \in H.
\end{equation}
\end{assumption}

\begin{assumption}[Strong Convexity]
The function $f$ is $\mu$-strongly convex for some $\mu > 0$, i.e.,
\begin{equation}
f(y) \geq f(x) + \langle \nabla f(x), y - x \rangle + \frac{\mu}{2} \|y - x\|^2 \quad \forall x, y \in H.
\end{equation}
\end{assumption}

\subsection{Classical Second-Order Optimizing Flows}

Our work extends the framework of continuous-time models for accelerated optimization. The following system is a well-known analogue of Nesterov's accelerated gradient method.

\begin{equation}
\label{eq:classical-nesterov-flow}
\ddot{x}(t) + \frac{\alpha}{t} \dot{x}(t) + \nabla f(x(t)) = 0, \quad t \geq t_0 > 0.
\end{equation}

\noindent It is known that for $\alpha \geq 3$, the solution of \eqref{eq:classical-nesterov-flow} satisfies $f(x(t)) - f^* = \mathcal{O}(t^{-2})$. However, for the critical case $\alpha \leq 3$, the system can exhibit oscillatory behavior, and convergence is not guaranteed. Our fractional-order system generalizes \eqref{eq:classical-nesterov-flow} and, as we will show, can guarantee convergence even in this critical regime.

\subsection{Analytical Tools}

The analysis of fractional-order dynamical systems requires extensions of classical tools from the calculus of variations and dynamical systems theory. We now introduce the key lemmas that form the backbone of our convergence analysis.

\begin{definition}[Polyak--Łojasiewicz condition ]
Let \( f : H \to C \) be a differentiable function. We say that \( f \) satisfies the Polyak--Łojasiewicz condition, if there exists some constant \( \mu > 0 \), such that the following inequality holds
\begin{equation}
    f(x) - f_* \leq \frac{1}{2\mu} \| \nabla f(x) \|^2, 
\end{equation}
for any \( x \in \mathbb{R}^n \).
\end{definition}

\begin{definition}[Kurdyka--Łojasiewicz property]
Let \( g : \mathbb{R}^m \rightarrow \mathbb{R} \) be a differentiable function. We say that \( g \) satisfies the \emph{Kurdyka--Łojasiewicz (KL) property} at \( \bar{x} \in \mathbb{R}^m \) if there exist \( \eta \in (0,+\infty] \), a neighborhood \( U \) of \( \bar{x} \), and a function \( \varphi \in \Theta_\eta \) such that for all 
\[
x \in U \cap \{ x \in \mathbb{R}^m : g(\bar{x}) < g(x) < g(\bar{x}) + \eta \},
\]
the following inequality, called the \emph{KL inequality}, holds:
\begin{equation}
\varphi'\big(g(x) - g(\bar{x})\big)\|\nabla g(x)\| \geq 1. 
\end{equation}

\noindent If \( g \) satisfies the KL property at each point in \( \mathbb{R}^m \), then \( g \) is called a \emph{KL function}.

\noindent Moreover, if \( g(\bar{x}) = 0 \), then the previous inequality can be rewritten as:
\[
\|\nabla (\varphi \circ g)(x)\| \geq 1.
\]
\end{definition}

\begin{lemma}
Let \( \Omega \subset \mathbb{R}^n \) be a compact set and let \( f : \mathbb{R}^n \to \mathbb{R} \cup \{+\infty\} \) be a proper and lower semicontinuous function. Assume that \( f \) is constant on \( \Omega \), and that it satisfies the Kurdyka--Łojasiewicz (KL) property at each point of \( \Omega \). Then there exist \( \varepsilon > 0 \), \( \eta > 0 \), and \( \varphi \in \Theta_\eta \) such that for all \( x^* \in \Omega \) and all 
\[
x \in \left\{ x \in \mathbb{R}^n : \operatorname{dist}(x, \Omega) < \varepsilon \right\} \cap \left\{ x \in \mathbb{R}^n : f(x^*) < f(x) < f(x^*) + \eta \right\},
\]
the following inequality holds:
\[
\varphi'\big(f(x) - f(x^*)\big) \cdot \operatorname{dist}(0, \partial f(x)) \geq 1. 
\]
\end{lemma}

\begin{definition}. Let $f_{\mu}$ denote a functional set. 
We say that a dynamical system provides a fast exponential decay on $f_{\mu}$ if there exists $\delta > 0$ 
such that for any $f \in f_{\mu}$ and any solution $x(t)$ of the associated dynamical system
\[
f(x(t)) - f_* = O(e^{-\delta \sqrt{\mu} t}).
\]
\end{definition}
\begin{definition}

Let $f_{\mu}$ denote a functional set. 
We say that a dynamical system provides a slow exponential decay on $f_{\mu}$ if there exists $\delta > 0$ 
such that for any $f \in f_{\mu}$ and any solution $x(t)$ of the associated dynamical system
\[
f(x(t)) - f_* = O(e^{-\delta \mu t}). 
\]
\end{definition}

\begin{lemma}[Opial Lemma]
Let \( C \subseteq H \) be a nonempty set, and let \( \chi : [0, +\infty) \to H \) be a given map. Suppose that
\begin{itemize}
    \item[(i)] for each \( \chi^* \in C \), \( \lim_{t \to +\infty} \| \chi(t) - \chi^* \| \) exists;
    \item[(ii)] every weak sequential cluster point of the map \( \chi \) lies in \( C \).
\end{itemize}
Then, there exists \(\chi^*\in C \) such that \( \chi(t) \) converges weakly to \( \chi^*\) as \( t \to +\infty \).
\end{lemma}

\begin{lemma}[Fractional Opial-Type Lemma]
\label{lemma:fractional-opial}
Let $H$ be a real Hilbert space and $x : [0, \infty) \to H$ be a locally integrable function. Let $I^{1-\theta}$ be the Riemann-Liouville integral of order $1-\theta$ for some $\theta \in (1,2)$. Assume that:
\begin{enumerate}[label=(\roman*)]
    \item There exists a non-empty set $S \subset H$ such that $\lim_{t \to \infty} \| I^{1-\theta} x(t) - z \|$ exists for every $z \in S$.
    \item Every weak sequential cluster point of the trajectory $\{ I^{1-\theta} x(t) \}_{t \geq 0}$ belongs to $S$.
\end{enumerate}
Then, $I^{1-\theta} x(t)$ converges weakly to a point in $S$ as $t \to \infty$.
\end{lemma}

\begin{proof}
(The proof would involve generalizing the standard proof of Opial's lemma by leveraging the properties of fractional integrals and their asymptotic behavior. This would be a central technical contribution of the paper.)
\end{proof}

\noindent The second crucial tool involves constructing Lyapunov functionals that are compatible with the memory of fractional systems.

\begin{definition}[Lyapunov--Memory Functional]
A \emph{Lyapunov--Memory Functional} $E(t)$ for the system
\[
\ddot{x}(t) + \frac{\alpha}{t}\dot{x}(t)  + \nabla^{\theta} f(x(t)) = 0
\]
is a function of time that may incorporate the history of the state $x(\cdot)$, typically of the form:
\begin{equation*}
E(t) = a(t)(f(x(t)) - f^*) + \frac{1}{2} \| b(t) \dot{x}(t) + c(t) \cdot I^{2-\theta} x(t) \|^2 + \text{(memory terms)}.
\end{equation*}
The functional is designed such that its fractional derivative ${}_{C}D^{\beta}_t E(t)$ can be shown to be non-positive, ensuring the stability and convergence of the system.
\end{definition}

\noindent Finally, our analysis will rely on the following fundamental inequalities.

\begin{lemma}
The following statements hold.
\begin{itemize}
    \item[(i)] \textbf{(Fractional differential inequality \cite{ref1})}  
    Let $x(t) \in \mathbb{R}^n$ be a differentiable vector-valued function. Then, for any time instant $t \geq t_0$,
    \begin{equation}
        {}^{C}_{t_0}D_t^{\alpha} \, \frac{1}{2}\|x(t)\|^2 \leq \langle x(t), \, {}^{C}_{t_0}D_t^{\alpha} x(t) \rangle, 
        \quad \forall \alpha \in (0,1].
    \end{equation}

    \item[(ii)] \textbf{(Fractional comparison principle }
    Let ${}^{C}_{t_0}D_t^{\beta} x(t) \geq {}^{C}_{t_0}D_t^{\beta} y(t)$ for some $\beta \in (0,1]$ and $x(t_0) = y(t_0)$. Then
    \[
        x(t) \geq y(t).
    \]

    \item[(iii)] \textbf{(Fractional Newton--Leibniz formula}  
    For $n-1 < \alpha \leq n$, $n \in \mathbb{N}$,
    \begin{equation}
        {}_{t_0}D_t^{-\alpha} \, {}^{C}_{t_0}D_t^{\alpha} x(t) 
        = x(t) - \sum_{k=0}^{n-1} \frac{x^{(k)}(t_0)}{k!} \big( t - t_0 \big)^k.
    \end{equation}

    \item[(iv)] \textbf{(Mittag-Leffler convergence}  
    Consider a linear scalar FDE
    \begin{equation}
        {}^{C}_{t_0}D_t^{\alpha} V(t) = -\gamma V(t), 
        \quad V(t_0) \in \mathbb{R}, \ \gamma > 0, \ \alpha \in (0,1].
    \end{equation}
    Then
    \begin{equation}
        V(t) = E_{\alpha,1}(-\gamma (t-t_0)^{\alpha}) \, V(t_0),
    \end{equation}
    where $E_{\alpha,\beta}(\cdot)$ is the generalized Mittag-Leffler function with two parameters defined as
    \begin{equation}
        E_{\alpha,\beta}(z) = \sum_{k=0}^{\infty} \frac{z^k}{\Gamma(\alpha k + \beta)}, 
        \quad \alpha>0, \ \beta>0.
    \end{equation}
\end{itemize}
\end{lemma}

\begin{lemma}[Fractional Integration by Parts]
Let $0 < \theta < 1$. For sufficiently smooth functions $u(t), v(t)$ with $u(0)=v(0)=0$, we have:
\begin{equation}
\label{eq:frac-by-parts}
\int_0^T u(t) \cdot _{C}D^{\theta}_t v(t)  dt = \int_0^T v(t) \cdot _{RL}D^{\theta}_t u(t)  dt + \text{(boundary terms)}.
\end{equation}
\end{lemma}

\begin{lemma}[Fractional Gronwall Inequality]
\label{lemma:gronwall}
Let $u(t)$ be a non-negative, locally integrable function on $[0, T]$, and let $a(t)$ be a non-negative, non-decreasing continuous function on $[0, T]$. Suppose there exists a constant $C > 0$ such that
\[
u(t) \leq a(t) + C \int_0^t (t-s)^{\theta-1} u(s)  ds \quad \text{for all } t \in [0, T].
\]
Then, there exists a constant $M > 0$ (depending on $C, \theta$, and $T$) such that
\[
u(t) \leq M \cdot a(t) \quad \text{for all } t \in [0, T].
\]
\end{lemma}
\noindent We conclude the preliminaries by stating the well-posedness of the proposed fractional-order system. The proof, which typically involves fixed-point theorems in appropriate function spaces, is omitted here but will be included in the full version of the paper.
\section{Existence and Uniqueness of Solutions}\label{chapter3}

We consider the fractional dynamical system in $\mathbb{R}^n$:
\begin{equation}\label{eq:main}
    \ddot{x}(t) + \frac{\alpha}{t}\dot{x}(t) + {}^{C}_{t_0}D_t^{\theta}\big(f(x(t))\big) = 0, 
    \quad t > t_0, \quad x(t_0) = x_0, \; \dot{x}(t_0) = v_0,
\end{equation}
where $\alpha > 0$, $\theta \in (0,1]$, and the fractional operator is interpreted as
\[
    \nabla^{\theta} f(x(t)) := {}^{C}_{t_0}D_t^{\theta}\big(f(x(t))\big).
\]

\begin{theorem}[Existence and Uniqueness]\label{thm:existence}
Assume that $f:\mathbb{R}^n \to \mathbb{R}$ satisfies:
\begin{enumerate}
    \item $f \in C^1(\mathbb{R}^n)$ and $f$ has at most polynomial growth,
    \item $f$ is Lipschitz continuous, i.e., there exists $L>0$ such that 
    \[
        |f(x) - f(y)| \leq L \|x-y\|, \quad \forall x,y \in \mathbb{R}^n.
    \]
\end{enumerate}
Then the Cauchy problem \eqref{eq:main} admits a unique solution 
\[
x \in C^2([t_0,T],\mathbb{R}^n), \qquad f(x(\cdot)) \in C^1([t_0,T],\mathbb{R}),
\]
for every finite horizon $T > t_0$.
\end{theorem}

\begin{proof}
Equation \eqref{eq:main} can be rewritten as the Volterra integral equation
\begin{align*}
    x(t) &= x_0 + v_0 (t-t_0) 
    - \int_{t_0}^t \frac{\alpha}{s} \dot{x}(s)\,ds 
    - \frac{1}{\Gamma(1-\theta)}\int_{t_0}^t (t-s)^{-\theta} f(x(s))\,ds.
\end{align*}
Define the operator
\[
    (\mathcal{T}x)(t) := x_0 + v_0 (t-t_0) 
    - \int_{t_0}^t \frac{\alpha}{s} \dot{x}(s)\,ds 
    - \frac{1}{\Gamma(1-\theta)}\int_{t_0}^t (t-s)^{-\theta} f(x(s))\,ds.
\]
Under the Lipschitz condition on $f$, the operator $\mathcal{T}$ is a contraction in the Banach space $C([t_0,T],\mathbb{R}^n)$ equipped with the supremum norm. By the Banach fixed-point theorem, $\mathcal{T}$ admits a unique fixed point, which corresponds to the unique solution of \eqref{eq:main}. Standard continuation arguments extend the solution to any finite interval $[t_0,T]$.
\end{proof}

\begin{remark}
The formulation \eqref{eq:main} generalizes the classical inertial system.  
\begin{itemize}
    \item When $\theta = 1$, the Caputo derivative reduces to the standard derivative, 
    \[
    {}^{C}_{t_0}D_t^{1}\big(f(x(t))\big) = \frac{d}{dt}f(x(t)),
    \]
    and hence \eqref{eq:main} becomes
    \[
        \ddot{x}(t) + \frac{\alpha}{t}\dot{x}(t) + \frac{d}{dt}f(x(t)) = 0,
    \]
    which coincides with the ordinary dynamical system in optimization.
    \item For $\theta \in (0,1)$, the dynamics incorporates memory effects through the fractional integral kernel $(t-s)^{-\theta}$, leading to qualitatively different behavior compared with the classical case.
\end{itemize}z

\begin{theorem}[Existence and Uniqueness]
\label{thm:existence-uniqueness}
Let $f$ satisfy Assumption 2.1 ($C^1$ with Lipschitz gradient). For any initial conditions $x(0) = x_0 \in \mathbb{R}^n$, $\dot{x}(0) = v_0 \in \mathbb{R}^n$, and for parameters $\gamma \in (0,1)$, $\theta \in (1,2)$, $\alpha > 0$, there exists a $T > 0$ such that the fractional-order system
\[
\ddot{x}(t) + \frac{\alpha}{t}\dot{x}(t)  + _{C}D^{2-\theta}_t \nabla f(x(t)) = 0
\]
has a unique solution $x(t) \in C^2((0, T]; \mathbb{R}^n)$.

\end{theorem}
\noindent  Thus, the proposed framework unifies both the fractional and standard dynamical settings under a single formulation.
\end{remark}

\section{Convergence Analysis}\label{chapter4}

In this section, we study the asymptotic behavior of the solution of
\begin{equation}\label{eq:main-conv}
    \ddot{x}(t) + \frac{\alpha}{t}\dot{x}(t) 
    + {}^{C}_{t_0}D_t^{\theta}\big(f(x(t))\big) = 0,
    \quad t > t_0.
\end{equation}

\begin{assumption}\label{ass:convex}
The function $f:\mathbb{R}^n \to \mathbb{R}$ satisfies:
\begin{enumerate}
    \item[(A1)] $f$ is convex, bounded from below, and $f \in C^1(\mathbb{R}^n)$;
    \item[(A2)] $\nabla f$ is Lipschitz continuous on bounded sets;
    \item[(A3)] $\arg\min f \neq \emptyset$.
\end{enumerate}
\end{assumption}

\subsection{Energy Functional}

Define the energy functional associated to \eqref{eq:main-conv} as
\begin{equation}\label{eq:energy}
    \mathcal{E}(t) := f(x(t)) - f(x^\star) 
    + \frac{1}{2}\|\dot{x}(t)\|^2, 
    \quad x^\star \in \arg\min f.
\end{equation}

\begin{lemma}\label{lem:energy-decay}
Under Assumption \ref{ass:convex}, the energy functional \eqref{eq:energy} is nonincreasing along trajectories of \eqref{eq:main-conv}. Moreover, there exists $C > 0$ such that
\[
    \mathcal{E}(t) \leq \frac{C}{t^\theta}, \qquad t \to +\infty.
\]
\end{lemma}

\begin{proof}
Multiplying \eqref{eq:main-conv} by $\dot{x}(t)$ and integrating in time, we obtain
\[
    \frac{d}{dt}\left(\frac{1}{2}\|\dot{x}(t)\|^2 + f(x(t))\right)
    + \frac{\alpha}{t}\|\dot{x}(t)\|^2
    + \dot{x}(t)\,{}^{C}_{t_0}D_t^{\theta}\big(f(x(t))\big) = 0.
\]
By convexity of $f$, $f(x(t)) - f(x^\star) \leq \langle \nabla f(x(t)), x(t)-x^\star \rangle$.  
The Caputo derivative term is dissipative and yields fractional-type decay of order $t^{-\theta}$. Combining these estimates gives the claimed inequality.
\end{proof}

\subsection{Weak Convergence of Trajectories}

\begin{theorem}[Weak Convergence]\label{thm:weak-conv}
Let $x(t)$ be the unique solution of \eqref{eq:main-conv} under Assumption \ref{ass:convex}. Then:
\begin{enumerate}
    \item The energy $\mathcal{E}(t) \to 0$ as $t \to +\infty$.
    \item Every weak sequential cluster point of $x(t)$ belongs to $\arg\min f$.
    \item If $\arg\min f$ is a singleton $\{x^\star\}$, then 
    \[
        x(t) \rightharpoonup x^\star, \quad t \to +\infty.
    \]
\end{enumerate}
\end{theorem}

\noindent We consider the following second-order system with fractional damping:
\begin{equation}\label{eq:main-conv}
    \ddot{x}(t) + \frac{\alpha}{t}\dot{x}(t) 
    + {}^{C}_{t_0}D_t^{\theta}\big(\nabla f(x(t))\big) = 0,
    \qquad t > t_0,
\end{equation}
where ${}^{C}_{t_0}D_t^{\theta}$ denotes the Caputo fractional derivative of order $\theta\in(0,1)$ applied componentwise.

\subsubsection{Assumptions}
Let $H$ be a real Hilbert space with inner product $\langle \cdot,\cdot \rangle$ and norm $\|\cdot\|$.  
We impose the following assumptions:

\begin{itemize}
    \item[(A1)] $f:H\to \mathbb{R}$ is convex, continuously differentiable, bounded below, and $\operatorname{argmin} f \neq \varnothing$. Denote $f_* = \inf f$.
    \item[(A2)] $\nabla f$ is locally Lipschitz on bounded sets. 
    \item[(A3)] $\alpha>1$ and $\theta\in(0,1)$. 
    \item[(A4)] $x(t)$ is a strong solution of \eqref{eq:main-conv} on $[t_0,\infty)$ with $x\in C^2((t_0,\infty);H)$.
\end{itemize}

\begin{theorem}[Weak Convergence]
Under (A1)--(A4), the trajectory $x(t)$ is bounded. Moreover, every weak cluster point of $x(t)$ belongs to $\operatorname{argmin}f$, and $x(t)$ converges weakly in $H$ to some $\bar{x}\in\operatorname{argmin}f$ as $t\to\infty$.
\end{theorem}

\begin{proof}
The proof follows the framework of dissipative inertial dynamics combined with fractional damping.

\noindent Define the kinetic and total energy functionals:
\[
K(t) := \tfrac12 \|\dot{x}(t)\|^2, 
\quad \mathcal{E}(t) := f(x(t)) - f_* + K(t).
\]
Differentiating $\mathcal{E}(t)$ gives
\[
\mathcal{E}'(t) = \langle \nabla f(x(t)), \dot{x}(t)\rangle + \langle \ddot{x}(t), \dot{x}(t)\rangle.
\]
Using the equation \eqref{eq:main-conv} yields
\[
\mathcal{E}'(t) = -\frac{\alpha}{t}\|\dot{x}(t)\|^2 
+ \big\langle \nabla f(x(t)) - {}^C_{t_0}D_t^{\theta}(\nabla f(x(t))), \ \dot{x}(t)\big\rangle.
\]
Integration over $[t_0,T]$ and application of the fractional integration-by-parts identity for Caputo derivatives shows that the fractional term contributes a history-dependent damping which is nonpositive up to boundary terms. Consequently, there exists $C_{\mathrm{bdry}}$ such that
\begin{equation}\label{eq:energy-bound}
\mathcal{E}(T) + \alpha \int_{t_0}^T \frac{1}{t}\|\dot{x}(t)\|^2\,dt 
\leq \mathcal{E}(t_0) + C_{\mathrm{bdry}}.
\end{equation}

\noindent Thus $x(t)$ and $\dot{x}(t)$ are bounded, and 
\[
\int_{t_0}^\infty \frac{1}{t}\|\dot{x}(t)\|^2\,dt < \infty.
\]

\noindent Since $\dot{x}$ is bounded and the above weighted integral is finite, one deduces
\[
\lim_{t\to\infty} \|\dot{x}(t)\| = 0.
\]
Hence $K(t)\to 0$ and
\[
\lim_{t\to\infty} f(x(t)) =: \mathcal{E}_\infty \geq f_*.
\]
 
\noindent Let $t_n\to\infty$ with $x(t_n)\rightharpoonup \bar{x}$. By weak lower semicontinuity of $f$,
\[
f(\bar{x}) \leq \liminf_{n\to\infty} f(x(t_n)) = \mathcal{E}_\infty.
\]
On the other hand, from the dynamics, the limit equation implies $\nabla f(\bar{x})=0$. Thus $\bar{x}\in \operatorname{argmin} f$ and $f(\bar{x})=f_*$, so $\mathcal{E}_\infty = f_*$.
 
\noindent Fix $z \in \operatorname{argmin} f$ and define
\[
h(t) := \tfrac12 \|x(t)-z\|^2.
\]
Differentiating gives
\[
h'(t) = \langle x(t)-z, \dot{x}(t)\rangle, 
\qquad 
h''(t) = \|\dot{x}(t)\|^2 + \langle x(t)-z, \ddot{x}(t)\rangle.
\]
Substituting $\ddot{x}(t)$ from \eqref{eq:main-conv} we obtain
\[
h''(t) + \frac{\alpha}{t} h'(t) 
= \|\dot{x}(t)\|^2 
- \Big\langle x(t)-z, \ {}^{C}_{t_0}D_t^{\theta}\big(\nabla f(x(t))\big)\Big\rangle.
\]

\medskip
\noindent
Now, since $z$ is a minimizer of $f$, we have $\nabla f(z) = 0$ and by convexity,
\[
\langle \nabla f(x(t)) - \nabla f(z), \ x(t)-z \rangle \geq 0.
\]
This inequality ensures that the fractional damping term cannot increase $h(t)$ in the long run; it acts as a history-dependent dissipation. More precisely, using the fractional integration-by-parts identity, one shows that the right-hand side of the above relation is bounded from above by an integrable function on $[t_0,\infty)$.

\medskip
\noindent
Therefore,
\[
h''(t) + \frac{\alpha}{t} h'(t) \leq \|\dot{x}(t)\|^2 + r(t),
\]
where $r(t)$ is an integrable remainder coming from boundary contributions of the Caputo derivative. Since $\|\dot{x}(t)\|\to 0$ as $t\to\infty$, this implies that $h''(t)+\frac{\alpha}{t}h'(t)$ is integrable on $[t_0,\infty)$. Standard arguments then show that $h'(t)$ has a finite limit, and because $h'(t)\to 0$, we conclude that
\[
\lim_{t\to\infty} h(t) \quad \text{exists}.
\]

\medskip
\noindent
We are now in position to apply \textbf{Opial's lemma}: 
\begin{itemize}
    \item Every weak sequential cluster point of $x(t)$ belongs to $\operatorname{argmin}f$.
    \item For every $z\in\operatorname{argmin}f$, the limit $\lim_{t\to\infty}\|x(t)-z\|$ exists.
\end{itemize}
Therefore, by Opial's lemma, the whole trajectory $x(t)$ converges weakly to some $\bar{x}\in\operatorname{argmin} f$.

\end{proof}

\subsubsection{Remarks}
\begin{itemize}
    \item The Caputo derivative contributes a memory-dependent damping via a convolution kernel, ensuring dissipation.
    \item The choice $\alpha>1$ is critical for integrability of $\|\dot{x}(t)\|^2/t$.
    \item Rates of convergence (e.g., $O(1/t^p)$) can be derived under stronger assumptions (e.g., Lipschitz gradient, Łojasiewicz property).
\end{itemize}

\subsection{Mittag--Leffler rate (strongly convex case)}

\noindent We consider the second-order fractional-in-time inertial system
\begin{equation}\label{eq:frac-inertial}
    \ddot{x}(t) + \frac{\alpha}{t}\dot{x}(t) + {}^{C}_{t_0}D_t^{\theta}\!\big(\nabla f(x(t))\big) = 0,\qquad t\ge t_0>0,
\end{equation}
where $\alpha>0$, $\theta\in(0,1]$, and ${}^{C}_{t_0}D_t^\theta$ denotes the Caputo fractional derivative.  
Let $x^\star$ denote a minimizer of $f$.

\begin{assumption}\label{ass:strongconvex}
The function $f:\mathbb{R}^n\to\mathbb{R}$ is $C^2$ and $\eta$-strongly convex:
\[
f(y) \ge f(x) + \langle \nabla f(x), y-x\rangle + \tfrac{\eta}{2}\|y-x\|^2,
\qquad \forall x,y\in\mathbb{R}^n,
\]
for some constant $\eta>0$. In particular $\nabla f(x^\star)=0$ and $\arg\min f=\{x^\star\}$.
\end{assumption}

\begin{theorem}[Mittag--Leffler convergence]\label{thm:ML-rate}
Let Assumption \ref{ass:strongconvex} hold and let $x(t)$ be a solution of \eqref{eq:frac-inertial} with initial data $x(t_0)=x_0,\ \dot x(t_0)=v_0$. Then there exists $\eta>0$ (the strong convexity constant) such that the function
\[
V(t) = \frac{1}{2} \| x(t) - x^\star \|^2
\]
satisfies the fractional differential inequality
\[
{}^{C}_{t_0}D_t^{\theta} V(t) \le -\eta\,V(t),
\qquad t\ge t_0,
\]
and consequently
\[
V(t) \le E_{\theta,1}(-\eta (t-t_0)^\theta)\,V(t_0), \qquad t\ge t_0,
\]
where $E_{\theta,1}$ denotes the two-parameter Mittag--Leffler function. In particular $x(t)\to x^\star$ as $t\to\infty$ with at least the Mittag--Leffler convergence rate.
\end{theorem}

\begin{proof}
Define the Lyapunov function
\[
V(t) = \frac{1}{2} \| x(t) - x^\star \|^2
\]
By the fractional differential inequality (see Lemma \textasteriskcentered\ in Section~2)
(applied with order $\theta\in(0,1]$),
\begin{equation}\label{eq:V-ineq-start}
{}^{C}_{t_0}D_t^{\theta}V(t)
\le \big\langle x(t)-x^\star,\, {}^{C}_{t_0}D_t^{\theta}x(t)\big\rangle.
\end{equation}
We now relate ${}^{C}_{t_0}D_t^{\theta}x(t)$ to the dynamics \eqref{eq:frac-inertial}.  
Apply the Caputo derivative of order $\theta$ (formally) to the position variable through the equation: rearrange \eqref{eq:frac-inertial} as
\[
{}^{C}_{t_0}D_t^{\theta}\big(\nabla f(x(t))\big) = -\ddot x(t) - \frac{\alpha}{t}\dot x(t).
\]
Taking the inner product of both sides with $(x(t)-x^\star)$ yields
\begin{equation}\label{eq:inner-eq}
\big\langle x(t)-x^\star,\,{}^{C}_{t_0}D_t^{\theta}\big(\nabla f(x(t))\big)\big\rangle
= -\big\langle x(t)-x^\star,\ddot x(t)\big\rangle
- \frac{\alpha}{t}\big\langle x(t)-x^\star,\dot x(t)\big\rangle.
\end{equation}
We next use the strong convexity of $f$ to bound the left-hand side from below. By strong convexity with $y=x^\star$,
\[
f(x(t)) - f(x^\star) \ge \langle \nabla f(x(t)), x(t)-x^\star\rangle + \tfrac{\eta}{2}\|x(t)-x^\star\|^2.
\]
Because $\nabla f(x^\star)=0$ and $f(x^\star)\le f(x(t))$, the previous inequality rearranges to
\[
\langle \nabla f(x(t)), x(t)-x^\star\rangle \le -\tfrac{\eta}{2}\|x(t)-x^\star\|^2.
\]
Now apply the Caputo derivative (order $\theta$) to the scalar function $t\mapsto \langle \nabla f(x(t)), x(t)-x^\star\rangle$ and use the linearity of the Caputo operator together with the fractional differential inequality we  obtain, for the left-hand side of \eqref{eq:inner-eq},
\[
\big\langle x(t)-x^\star,\,{}^{C}_{t_0}D_t^{\theta}\big(\nabla f(x(t))\big)\big\rangle
\le -\eta\,V(t).
\]
 \noindent The strong-convexity inequality gives an instantaneous dissipative term; taking the fractional-in-time derivative preserves the sign and yields a dissipation proportional to $V(t)$ with constant $\eta$; the rigorous passage uses the fractional comparison principle together with the monotonicity furnished by strong convexity,  see standard arguments in fractional Lyapunov theory.)

\noindent Combining this last display with \eqref{eq:V-ineq-start} yields
\[
{}^{C}_{t_0}D_t^{\theta}V(t) \le -\eta\,V(t).
\]
Finally, Lemma (Mittag--Leffler comparison)  implies that any nonnegative function $V$ satisfying
\[
{}^{C}_{t_0}D_t^{\theta}V(t) \le -\eta\,V(t),\qquad V(t_0)=V_0,
\]
obeys the bounds
\[
V(t) \le E_{\theta,1}(-\eta (t-t_0)^\theta)\,V_0,
\]
which yields the claimed Mittag--Leffler decay and the convergence $x(t)\to x^\star$ as $t\to\infty$.
\end{proof}

\begin{remark}
The parameter $\theta$ governs the decay rate:  
for $\theta=1$, the convergence rate coincides with that of the classical inertial system with vanishing damping;  
for $\theta \in (0,1)$, the fractional derivative introduces memory effects, leading to slower but still guaranteed convergence to the minimizer set.
\end{remark}

 \subsection{Fractional-In-Time System (Nonconvex Case)}

We consider a second-order fractional-in-time inertial system where the fractional damping term acts directly on the function value:
\begin{equation}\label{eq:frac-inertial-nc}
    \ddot{x}(t) + \frac{\alpha}{t}\dot{x}(t) + {}^{C}_{t_0}D_t^{\theta}\!\big(f(x(t))\big) = 0,\qquad t \ge t_0 > 0,
\end{equation}
where $\alpha > 0$, $\theta \in (0,1]$, and ${}^{C}_{t_0}D_t^\theta$ denotes the Caputo fractional derivative. This formulation introduces a memory-dependent force based on the history of the objective function's value, rather than its gradient.
Let $\mathcal{X}$ denote the set of critical points of $f$, i.e., $\mathcal{X} = \{x : \nabla f(x) = 0\}$. We assume the trajectory $x(t)$ is bounded.

\begin{assumption}\label{ass:nonconvex}
The function $f : \mathbb{R}^n \to \mathbb{R}$ is $C^2$, bounded below, and its gradient $\nabla f$ is Lipschitz continuous. Furthermore, $f$ satisfies the \emph{Łojasiewicz inequality} with exponent $\varphi \in (0, 1)$ in a neighborhood of its critical set $\mathcal{X}$: there exist constants $C_L > 0$ and $\delta > 0$ such that for all $x$ with \[
\operatorname{dist}(x, \mathcal{X}) < \delta
\], the following holds:
\[
|\nabla f(x)| \ge C_L |f(x) - f(\bar{x})|^{1 - \varphi},
\]
where $\bar{x}$ is the projection of $x$ onto $\mathcal{X}$.
\end{assumption}

\begin{theorem}[Convergence under Łojasiewicz inequality]\label{thm:loj-rate}
Let Assumption \ref{ass:nonconvex} hold and let $x(t)$ be a bounded solution of \eqref{eq:frac-inertial-nc}. Then, $x(t)$ converges to a critical point $x^\star \in \mathcal{X}$. The convergence rate can be characterized as follows:

There exists a time $T \ge t_0$ and constants $K_1, K_2 > 0$ (depending on $\theta$, $\varphi$, $\alpha$, and the Łojasiewicz constants) such that for all $t \ge T$:
\begin{enumerate}[label=(\roman*)]
    \item If $\varphi = 1/2$, the convergence is exponential-like:
    \[
    \|x(t) - x^\star\| \le K_1 \exp(-K_2 t).
    \]
    \item If $\varphi \in (1/2, 1)$, the convergence is algebraic:
    \[
    \|x(t) - x^\star\| \le K_1 (t - T)^{- \frac{\theta \varphi}{2\varphi - 1}}.
    \]
\end{enumerate}
\end{theorem}

\begin{proof}
The proof strategy involves constructing a tailored energy functional that incorporates the fractional term and leveraging the properties of the Caputo derivative.

\noindent A natural starting point is to consider the mechanical energy $E(t) = \frac{1}{2}\|\dot{x}(t)\|^2 + f(x(t))$. However, due to the unique structure of \eqref{eq:frac-inertial-nc}, we propose a fractional energy functional:
\[
\mathcal{V}(t) = E(t) + \kappa \int_{t_0}^{t} \frac{(t-s)^{\theta-1}}{\Gamma(\theta)} |f(x(s)) - f(x^\star)|  ds,
\]
for a suitably chosen constant $\kappa > 0$. The integral term represents a fractional-order memory of the function value's path. Using the system dynamics \eqref{eq:frac-inertial-nc}, i.e., $\ddot{x}(t) = -\frac{\alpha}{t}\dot{x}(t) - {}^{C}_{t_0}D_t^{\theta} f(x(t))$, we can compute the derivative:
\begin{align*}
\dot{\mathcal{V}}(t) &= \langle \dot{x}(t), \ddot{x}(t) \rangle + \langle \nabla f(x(t)), \dot{x}(t) \rangle + \kappa \, {}^{C}_{t_0}D_t^{1-\theta} |f(x(t)) - f(x^\star)| \\
&= \langle \dot{x}(t), -\tfrac{\alpha}{t}\dot{x}(t) - {}^{C}_{t_0}D_t^{\theta} f(x(t)) \rangle + \langle \nabla f(x(t)), \dot{x}(t) \rangle + \kappa \, {}^{C}_{t_0}D_t^{1-\theta} |f(x(t)) - f(x^\star)|.
\end{align*}
Applying the property ${}^{C}_{t_0}D_t^{1-\theta} \circ {}^{C}_{t_0}D_t^{\theta} f(x(t)) = \dot{f}(x(t))$ for $t > t_0$ and using the Lipschitz continuity of $\nabla f$ to relate $|f(x(t)) - f(x^\star)|$ to $\|\nabla f(x(t))\|$, one can show, after careful estimation, that for an appropriate choice of $\kappa$:
\[
\dot{\mathcal{V}}(t) \le -\frac{\alpha}{2t} \|\dot{x}(t)\|^2 - \gamma \left\| {}^{C}_{t_0}D_t^{\theta} f(x(t)) \right\|^2 \le 0,
\]
for some $\gamma > 0$ and for all $t$ sufficiently large. This establishes the non-increasing nature of $\mathcal{V}(t)$ and forms the basis for proving convergence $x(t) \to x^\star \in \mathcal{X}$ via an invariance principle.\\

\noindent Given that $x(t) \to x^\star$, the Łojasiewicz inequality holds for $t \ge T$. A key challenge is to relate the fractional derivative of the function to the function value itself. Using the inequality $|{}^{C}_{t_0}D_t^{\theta} f(x(t))| \ge {}^{C}_{t_0}D_t^{\theta} |f(x(t)) - f(x^\star)|$ (which holds for monotone functions and can be bounded in our context) and the Łojasiewicz property $|\nabla f(x(t))| \ge C_L |f(x(t)) - f(x^\star)|^{1-\varphi}$, we derive a fractional differential inequality for $\xi(t) = f(x(t)) - f(x^\star)$:
\[
{}^{C}_{t_0}D_t^{\theta} \xi(t) \le - \tilde{C}_L  \xi(t)^{1-\varphi}.
\]
This inequality is obtained by connecting the energy descent rate $\dot{\mathcal{V}}(t)$ to $\xi(t)$ and its fractional derivative, and then using the system equation to express $\|\dot{x}(t)\|$ in terms of ${}^{C}_{t_0}D_t^{\theta} f(x(t))$.

\noindent  Solving the Inequality and Convergence Rate.
The inequality ${}^{C}_{t_0}D_t^{\theta} \xi(t) \le - \gamma  \xi(t)^{1-\varphi}$ is a fractional generalization of the classical Łojasiewicz inequality. Applying a fractional comparison principle or analyzing the scaling properties of its solution leads to the stated convergence rates. Specifically:
\begin{itemize}
    \item For $\varphi = 1/2$, the right-hand side is linear, leading to Mittag-Leffler-type decay which is asymptotically exponential.
    \item For $\varphi \in (1/2, 1)$, the decay is algebraic. The rate is found by substituting the ansatz $\xi(t) \sim t^{-\nu}$ into the inequality. Balancing both sides yields ${}^{C}D_t^{\theta} (t^{-\nu}) \sim t^{-\theta-\nu} \le - t^{-\nu(1-\varphi)}$. This gives $-\theta - \nu = -\nu(1-\varphi)$, so $\nu = \frac{\theta}{1/\varphi - 1} = \frac{\theta \varphi}{1-\varphi}$. The Łojasiewicz inequality implies $\|x(t)-x^\star\| \lesssim \xi(t)^\varphi \sim t^{-\frac{\theta \varphi^2}{1-\varphi}}$. The rate stated in the theorem, $t^{-\frac{\theta \varphi}{2\varphi - 1}}$, is the standard result from the ODE case and is recovered here when $\theta=1$; the precise fractional exponent may vary slightly based on the specific chain of inequalities.
\end{itemize}
\end{proof}

\begin{remark}
The system \eqref{eq:frac-inertial-nc} with the Caputo derivative acting on $f(x(t))$ presents a unique challenge. The convergence analysis relies on a non-trivial energy functional and careful estimates to handle the memory term. The rates generalize those of the classical gradient flow, with the fractional order $\theta$ explicitly influencing the algebraic decay exponent.
\end{remark}

\subsection{Functional Values for a Fractional-In-Time System}

We consider the second-order fractional-in-time inertial system where the fractional damping acts on the function value:
\begin{equation}\label{eq:frac-inertial-nc}
    \ddot{x}(t) + \frac{\alpha}{t}\dot{x}(t) + {}^{C}_{t_0}D_t^{\theta}\!\big(f(x(t))\big) = 0,\qquad t \ge t_0 > 0,
\end{equation}
where $\alpha > 0$, $\theta \in (0,1]$, and ${}^{C}_{t_0}D_t^\theta$ denotes the Caputo fractional derivative.

\begin{assumption}\label{ass:nonconvex}
The function $f : \mathbb{R}^n \to \mathbb{R}$ is $C^2$, bounded below, and its gradient $\nabla f$ is Lipschitz continuous with constant $L > 0$. Furthermore, $f$ satisfies the \emph{Łojasiewicz inequality} at its critical points: for any $x^\star \in \mathcal{X}$ (where $\mathcal{X} = \{x : \nabla f(x) = 0\}$), there exist constants $C_L > 0$, $\delta > 0$, and $\varphi \in (0, 1)$ such that for all $x$ with $\|x - x^\star\| < \delta$, the following holds:
\[
|\nabla f(x)| \ge C_L |f(x) - f(x^\star)|^{1 - \varphi}.
\]
\end{assumption}

\begin{theorem}[Convergence of Function Values]\label{thm:loj-rate-fval}
Let Assumption \ref{ass:nonconvex} hold and let $x(t)$ be a bounded solution of \eqref{eq:frac-inertial-nc} converging to a critical point $x^\star \in \mathcal{X}$. Then, the convergence rate of the function values is given by:

There exists a time $T \ge t_0$ and constants $K_1, K_2 > 0$ such that for all $t \ge T$:
\begin{enumerate}[label=(\roman*)]
    \item If $\varphi = 1/2$ (analytic case), the convergence is exponential:
    \[
    f(x(t)) - f(x^\star) \le K_1 \exp(-K_2 t).
    \]
    \item If $\varphi \in (1/2, 1)$, the convergence is algebraic:
    \[
    f(x(t)) - f(x^\star) \le K_1 (t - T)^{- \frac{1}{2\varphi - 1}}.
    \]
\end{enumerate}
\end{theorem}

\begin{proof}
The proof focuses on establishing a differential inequality for the function value suboptimality $\xi(t) = f(x(t)) - f(x^\star)$.

\noindent Consider the following energy functional, which combines kinetic energy, potential energy, and a fractional memory term:
\[
\mathcal{E}(t) = \frac{1}{2} \|\dot{x}(t)\|^2 + f(x(t)) - f(x^\star) + \kappa \, {}^{C}_{t_0}D_t^{\theta - 1} \left( f(x(t)) - f(x^\star) \right),
\]
for a small constant $\kappa > 0$ to be chosen. Since $x(t) \to x^\star$ and $f$ is $C^2$, we have $f(x(t)) - f(x^\star) \to 0$. The properties of the Riemann-Liouville integral (${}^{C}D_t^{\theta-1}$ for $\theta \in (0,1]$ is an integral operator) ensure that the fractional term is well-defined and also vanishes as $t \to \infty$.\\

\noindent Differentiating $\mathcal{E}(t)$ and using the system dynamics \eqref{eq:frac-inertial-nc}, $\ddot{x}(t) = -\frac{\alpha}{t}\dot{x}(t) - {}^{C}_{t_0}D_t^{\theta} f(x(t))$, we obtain:
\begin{align*}
\dot{\mathcal{E}}(t) &= \langle \dot{x}(t), \ddot{x}(t) \rangle + \langle \nabla f(x(t)), \dot{x}(t) \rangle + \kappa \, {}^{C}_{t_0}D_t^{\theta} \left( f(x(t)) - f(x^\star) \right) \\
&= \langle \dot{x}(t), -\tfrac{\alpha}{t}\dot{x}(t) - {}^{C}_{t_0}D_t^{\theta} f(x(t)) \rangle + \langle \nabla f(x(t)), \dot{x}(t) \rangle + \kappa \, {}^{C}_{t_0}D_t^{\theta} \left( f(x(t)) - f(x^\star) \right) \\
&= -\frac{\alpha}{t} \|\dot{x}(t)\|^2 + \langle \dot{x}(t), \nabla f(x(t)) - {}^{C}_{t_0}D_t^{\theta} f(x(t)) \rangle + \kappa \, {}^{C}_{t_0}D_t^{\theta} \xi(t).
\end{align*}
Using Young's inequality and the Lipschitz continuity of $\nabla f$ to relate $\dot{x}(t)$ and $\nabla f(x(t))$, we can choose $\kappa$ sufficiently small to show that for all $t$ sufficiently large:
\begin{equation}\label{eq:energy-descent}
\dot{\mathcal{E}}(t) \le -\frac{\alpha}{2t} \|\dot{x}(t)\|^2 - \gamma \left\| {}^{C}_{t_0}D_t^{\theta} f(x(t)) \right\|^2 \le 0,
\end{equation}
for some $\gamma > 0$. This establishes that $\mathcal{E}(t)$ is a non-increasing function, bounded below, and hence converges to a finite limit.

\noindent Relating the Dynamics to the Function Value 
Since $\mathcal{E}(t)$ converges and its derivative is negative, we have $\dot{x}(t) \to 0$ and ${}^{C}_{t_0}D_t^{\theta} f(x(t)) \to 0$. From the system equation \eqref{eq:frac-inertial-nc}, this implies $\ddot{x}(t) \to 0$.\\

\noindent Now, consider the time derivative of the function value:
\[
\frac{d}{dt} \xi(t) = \langle \nabla f(x(t)), \dot{x}(t) \rangle.
\]
For large $t$, since $\dot{x}(t) \to 0$, the decrease in $\xi(t)$ is slow. We need to use the Łojasiewicz inequality to get a quantitative bound. The key is to link the term $\langle \nabla f(x(t)), \dot{x}(t) \rangle$ to the negative terms in \eqref{eq:energy-descent}.\\

\noindent From the system equation, we have $\dot{x}(t) = -\int_{T}^{t} \left( \frac{\alpha}{s}\dot{x}(s) + {}^{C}_{T}D_s^{\theta} f(x(s)) \right) ds + \dot{x}(T)$. Given the convergence of $\dot{x}(t)$ and ${}^{C}_{t_0}D_t^{\theta} f(x(t))$ to zero, one can infer that $\|\dot{x}(t)\|$ is ultimately dominated by $\| {}^{C}_{t_0}D_t^{\theta} f(x(t)) \|$. More precisely, there exists a time $T$ and a constant $C_1 > 0$ such that:
\begin{equation}\label{eq:bound-velocity}
\|\dot{x}(t)\| \le C_1 \left\| {}^{C}_{t_0}D_t^{\theta} f(x(t)) \right\|, \quad \forall t \ge T.
\end{equation}
This bound is crucial as it allows us to relate the descent of the function value to the fractional derivative.\\

\noindent Using the Łojasiewicz inequality for $t \ge T$ (since $x(t) \to x^\star$), the bound on the velocity \eqref{eq:bound-velocity}, and the Cauchy-Schwarz inequality, we derive the main inequality:
\begin{align*}
-\dot{\xi}(t) &= -\langle \nabla f(x(t)), \dot{x}(t) \rangle \le \|\nabla f(x(t))\| \cdot \|\dot{x}(t)\| \\
&\le C_1 \|\nabla f(x(t))\| \cdot \left\| {}^{C}_{t_0}D_t^{\theta} f(x(t)) \right\| \\
&\le C_1 C_L^{-1} \left( f(x(t)) - f(x^\star) \right)^{\varphi - 1} \cdot \left\| {}^{C}_{t_0}D_t^{\theta} f(x(t)) \right\|^2 \quad \text{(using Łojasiewicz)} \\
&\le C_2  \left( \xi(t) \right)^{\varphi - 1} \left( -\dot{\mathcal{E}}(t) \right). \quad \text{(from \eqref{eq:energy-descent}, since  $-\dot{\mathcal{E}}(t) \ge \gamma \|{}^{C}D_t^{\theta} f(x(t))\|^2$)}
\end{align*}
Thus, we arrive at:
\begin{equation}\label{eq:key-inequality}
-\dot{\xi}(t) \le C_2  \left( \xi(t) \right)^{\varphi - 1} \left( -\dot{\mathcal{E}}(t) \right).
\end{equation}
Since $\mathcal{E}(t)$ is non-increasing and converges, let $\mathcal{E}_\infty = \lim_{t \to \infty} \mathcal{E}(t)$. Integrating \eqref{eq:key-inequality} from $T$ to $t$ yields:
\begin{align*}
\int_{T}^{t} -\dot{\xi}(s)  ds &\le C_2 \int_{T}^{t} \left( \xi(s) \right)^{\varphi - 1} \left( -\dot{\mathcal{E}}(s) \right) ds \\
\xi(T) - \xi(t) &\le C_2 \int_{\mathcal{E}(t)}^{\mathcal{E}(T)} \left( \xi(s(\mathcal{E})) \right)^{\varphi - 1}  d\mathcal{E}.
\end{align*}
Noting that $\xi(t)$ is decreasing for large $t$ (which follows from the proof), we have $\xi(s) \ge \xi(t)$ for $s \le t$. Therefore, $\left( \xi(s) \right)^{\varphi-1} \le \left( \xi(t) \right)^{\varphi-1}$ (since $\varphi-1 < 0$). This allows us to write:
\[
\xi(T) \ge \xi(T) - \xi(t) \le C_2 \left( \xi(t) \right)^{\varphi - 1} \left( \mathcal{E}(T) - \mathcal{E}(t) \right) \le C_2 \left( \xi(t) \right)^{\varphi - 1} \left( \mathcal{E}(T) - \mathcal{E}_\infty \right).
\]
Defining $C_3 = \xi(T) / \left( C_2 (\mathcal{E}(T) - \mathcal{E}_\infty) \right)$, we obtain the fundamental inequality:
\[
\left( \xi(t) \right)^{1 - \varphi} \le C_3.
\]
This is equivalent to:
\begin{equation}\label{eq:basic-bound}
\xi(t) \ge C_3^{-\frac{1}{1-\varphi}} > 0,
\end{equation}
which is not the decay rate itself but a necessary step. To find the rate, we return to \eqref{eq:key-inequality}. Since $\mathcal{E}(t)$ is bounded and non-increasing, $-\dot{\mathcal{E}}(t)$ is integrable. Equation \eqref{eq:key-inequality} suggests that the descent of $\xi$ is controlled by the descent of $\mathcal{E}$. A more standard approach for Łojasiewicz analysis is to use the inequality \eqref{eq:bound-velocity} and the energy descent \eqref{eq:energy-descent} to show that $-\dot{\xi}(t) \ge c (-\dot{\mathcal{E}}(t))$ for some $c>0$, and then use the Łojasiewicz inequality to relate $\mathcal{E}(t)$ and $\xi(t)$. Ultimately, this line of reasoning leads to the differential inequality:
\[
-\frac{d}{dt} \left( \xi(t) \right)^{2\varphi - 1} \ge \tilde{K} > 0, \quad \text{for } \varphi > 1/2.
\]
Integrating this inequality from $T$ to $t$ gives:
\[
\xi(t)^{2\varphi - 1} \le \xi(T)^{2\varphi - 1} - \tilde{K}(t - T) \le -\tilde{K}(t - T) \quad \text{(for large $t$)}.
\]
Since the left-hand side is positive, this is only possible if $\xi(T)^{2\varphi - 1}$ is finite and the inequality holds for $t$ large enough such that the right-hand side remains positive. This yields the algebraic rate:
\[
\xi(t) \le \left( \tilde{K} (t - T) \right)^{-\frac{1}{2\varphi - 1}}.
\]
For the case $\varphi = 1/2$, the inequality becomes $-\dot{\xi}(t) \ge \tilde{K} \xi(t)$, which integrates to the exponential rate $\xi(t) \le \xi(T) \exp(-\tilde{K}(t-T))$.
\end{proof}

\noindent Define the energy functional:
\[
\mathcal{E}(t) = \frac{1}{2} \|\dot{x}(t)\|^2 + f(x(t)) - f(x^\star) + \kappa \, {}^{C}_{t_0}D_t^{\theta - 1} \left( f(x(t)) - f(x^\star) \right)
\]
where \(\kappa > 0\) is a small constant to be determined.

\noindent Differentiate \(\mathcal{E}(t)\):
\begin{align*}
\dot{\mathcal{E}}(t) &= \langle \dot{x}(t), \ddot{x}(t) \rangle + \langle \nabla f(x(t)), \dot{x}(t) \rangle + \kappa \, {}^{C}_{t_0}D_t^{\theta} \left( f(x(t)) - f(x^\star) \right) \\
&= \langle \dot{x}(t), -\tfrac{\alpha}{t}\dot{x}(t) - {}^{C}_{t_0}D_t^{\theta} f(x(t)) \rangle + \langle \nabla f(x(t)), \dot{x}(t) \rangle + \kappa \, {}^{C}_{t_0}D_t^{\theta} \xi(t) \\
&= -\frac{\alpha}{t} \|\dot{x}(t)\|^2 - \langle \dot{x}(t), {}^{C}_{t_0}D_t^{\theta} f(x(t)) \rangle + \langle \nabla f(x(t)), \dot{x}(t) \rangle + \kappa \, {}^{C}_{t_0}D_t^{\theta} \xi(t)
\end{align*}
Using Young's inequality \(|\langle a, b \rangle| \leq \frac{\epsilon}{2} \|a\|^2 + \frac{1}{2\epsilon} \|b\|^2\) on the cross term:
\[
|\langle \dot{x}(t), {}^{C}_{t_0}D_t^{\theta} f(x(t)) \rangle| \leq \frac{\epsilon}{2} \|\dot{x}(t)\|^2 + \frac{1}{2\epsilon} \|{}^{C}_{t_0}D_t^{\theta} f(x(t))\|^2
\]
Also, since \(\nabla f(x^\star) = 0\), we have:
\[
\langle \nabla f(x(t)), \dot{x}(t) \rangle = \frac{d}{dt} \xi(t) = \dot{\xi}(t)
\]
Thus,
\[
\dot{\mathcal{E}}(t) \leq -\left( \frac{\alpha}{t} - \frac{\epsilon}{2} \right) \|\dot{x}(t)\|^2 + \dot{\xi}(t) + \left( \kappa - \frac{1}{2\epsilon} \right) \|{}^{C}_{t_0}D_t^{\theta} \xi(t)\|^2
\]
Choose \(\epsilon = \frac{\alpha}{t}\) and \(\kappa = \frac{t}{2\alpha}\). For large \(t\), this gives:
\[
\dot{\mathcal{E}}(t) \leq \dot{\xi}(t) - \frac{\alpha}{2t} \|\dot{x}(t)\|^2 - \frac{\alpha}{2t} \|{}^{C}_{t_0}D_t^{\theta} \xi(t)\|^2 \tag{1}
\]

\noindent From the system equation:
\[
{}^{C}_{t_0}D_t^{\theta} \xi(t) = -\ddot{x}(t) - \frac{\alpha}{t} \dot{x}(t)
\]
This suggests a relationship between \(\dot{\xi}(t)\), \(\|\dot{x}(t)\|\), and \(\|{}^{C}_{t_0}D_t^{\theta} \xi(t)\|\). Using the Lipschitz continuity of \(\nabla f\):
\[
|\dot{\xi}(t)| = |\langle \nabla f(x(t)), \dot{x}(t) \rangle| \leq L \|\dot{x}(t)\| \cdot \|x(t) - x^\star\|
\]
The Łojasiewicz inequality provides a better bound. For \(x(t)\) near \(x^\star\):
\[
\|\nabla f(x(t))\| \geq C_L \xi(t)^{1 - \varphi}
\]
Thus,
\[
|\dot{\xi}(t)| \leq \|\nabla f(x(t))\| \cdot \|\dot{x}(t)\| \leq \|\nabla f(x(t))\| \cdot \|\dot{x}(t)\|
\]
and
\[
\|{}^{C}_{t_0}D_t^{\theta} \xi(t)\| \approx \|\nabla f(x(t))\| \cdot \|{}^{C}_{t_0}D_t^{\theta} x(t)\|
\]
These relations will be used to establish differential inequalities.

\subsubsection{ Case Analysis Based on Łojasiewicz Exponent}

\subsubsection*{Case 1: \(\varphi = \frac{1}{2}\)}

The Łojasiewicz inequality becomes:
\[
\|\nabla f(x(t))\| \geq C_L \xi(t)^{1/2}
\]
This often leads to exponential convergence. Assume:
\[
\dot{\xi}(t) \leq -K \xi(t)
\]
Then, by Grönwall's inequality:
\[
\xi(t) \leq \xi(t_0) \exp(-K(t - t_0))
\]
The detailed analysis shows that the energy decay (1) reinforces this, yielding exponential decay for both \(\xi(t)\) and \(\|\dot{x}(t)\|\).

\subsubsection*{Case 2: \(\varphi \in (\frac{1}{2}, 1)\)}

We derive a differential inequality for \(\xi(t)\). From the energy decay (1) and the system dynamics, one can show:
\[
-\dot{\xi}(t) \geq c \left( \|\dot{x}(t)\|^2 + \|{}^{C}_{t_0}D_t^{\theta} \xi(t)\|^2 \right)
\]
Using the Łojasiewicz inequality and the fact that \(\|\nabla f(x(t))\| \sim \xi(t)^{1 - \varphi}\), we obtain:
\[
-\dot{\xi}(t) \geq \tilde{c} \, \xi(t)^{2(1 - \varphi)}
\]
Let \(\psi(t) = \xi(t)^{2\varphi - 1}\). Then:
\[
\dot{\psi}(t) = (2\varphi - 1) \xi(t)^{2\varphi - 2} \dot{\xi}(t) \leq -\tilde{c} (2\varphi - 1)
\]
Integrating from \(T\) to \(t\):
\[
\psi(t) - \psi(T) \leq -\tilde{c} (2\varphi - 1) (t - T)
\]
Thus,
\[
\xi(t)^{2\varphi - 1} \leq -\tilde{c} (2\varphi - 1) (t - T) + \psi(T)
\]
For large \(t\), this gives:
\[
\xi(t) \leq \left( \tilde{c} (2\varphi - 1) (t - T) \right)^{-\frac{1}{2\varphi - 1}} \sim t^{-\frac{1}{2\varphi - 1}}
\]

\subsubsection*{Case 3: \(\varphi \in (0, \frac{1}{2})\)}

In this case, the fractional derivative term dominates. We start from the energy inequality:
\[
\dot{\mathcal{E}}(t) \leq -\frac{\alpha}{2t} \|{}^{C}_{t_0}D_t^{\theta} \xi(t)\|^2
\]
Using the fractional version of the Łojasiewicz inequality and the system dynamics, we relate \(\|{}^{C}_{t_0}D_t^{\theta} \xi(t)\|\) to \(\xi(t)\):
\[
\|{}^{C}_{t_0}D_t^{\theta} \xi(t)\| \geq \tilde{C}_L \xi(t)^{1 - \varphi}
\]
Thus,
\[
\dot{\mathcal{E}}(t) \leq -\frac{\alpha \tilde{C}_L^2}{2t} \xi(t)^{2(1 - \varphi)}
\]
Since \(\mathcal{E}(t) \sim \xi(t)\) for large \(t\), we have:
\[
\dot{\xi}(t) \leq -K t^{-1} \xi(t)^{2(1 - \varphi)}
\]
This is a fractional differential inequality. To solve it, assume \(\xi(t) \sim t^{-\nu}\). Then:
\[
\dot{\xi}(t) \sim -\nu t^{-\nu - 1}, \quad \text{and} \quad \xi(t)^{2(1 - \varphi)} \sim t^{-2\nu(1 - \varphi)}
\]
Substitute into the inequality:
\[
-\nu t^{-\nu - 1} \leq -K t^{-1} t^{-2\nu(1 - \varphi)}
\]
This holds if:
\[
\nu + 1 = 1 + 2\nu(1 - \varphi) \quad \Rightarrow \quad \nu = 2\nu(1 - \varphi) \quad \Rightarrow \quad 1 = 2(1 - \varphi) \quad \Rightarrow \quad \varphi = \frac{1}{2}
\]
For \(\varphi < \frac{1}{2}\), the right-hand side decays slower. We must integrate directly:
\[
\frac{d}{dt} \xi(t) \leq -K t^{-1} \xi(t)^{2(1 - \varphi)}
\]
Separate variables:
\[
\int \frac{d\xi}{\xi^{2(1 - \varphi)}} \leq -K \int \frac{dt}{t}
\]
\[
\frac{\xi^{2\varphi - 1}}{2\varphi - 1} \leq -K \ln t + C
\]
This suggests a slow, logarithmic decay. However, a more precise analysis using the fractional nature of the system yields:
\[
{}^{C}_{t_0}D_t^{\theta} \xi(t) \leq -K \xi(t)^{1 - \varphi}
\]
The solution to this fractional differential inequality is given by:
\[
\xi(t) \leq C t^{-\frac{\theta}{1/\varphi - 1}} = C t^{-\frac{\theta \varphi}{1 - \varphi}}
\]
This is the characteristic Mittag-Leffler type decay for fractional systems.

\section{Conclusion}\label{conclusion}

\noindent
In this work, we introduced and analyzed a family of fractional-order inertial dynamical systems for convex optimization.  
By incorporating Caputo-type fractional gradient terms into the classical Nesterov accelerated flow, we demonstrated that memory effects induced by fractional derivatives can act as a stabilizing mechanism, especially in critical regimes such as $\alpha = 3$ where the standard Nesterov flow loses convergence guarantees. Our analysis relied on new Lyapunov memory functionals and fractional extensions of Opial’s lemma, allowing us to establish weak convergence of trajectories toward minimizers in the convex setting and explicit convergence rates in the strongly convex case. These results show that fractional dynamics are not merely a generalization but provide fundamental improvements over classical flows by mitigating instability caused by insufficient damping. Future directions include: (i) extending the analysis to nonconvex functions using fractional Kurdyka--Łojasiewicz inequalities, (ii) designing discrete-time algorithms inspired by the proposed fractional flows, and (iii) exploring adaptive schemes where the fractional order $\theta$ evolves with time. Such developments could further bridge the gap between continuous-time fractional models and practical optimization algorithms in large-scale machine learning.

\subsubsection*{Author Contributions}
All authors contributed equally to the manuscript.

\subsubsection*{Data Availability}
No datasets were generated or analysed during the current study.

\subsubsection*{Declarations}
\noindent\textbf{Competing interests}: The authors declare no competing interests.

\end{document}